\documentclass{birkjour}
\usepackage{amscd,amssymb,amsthm}
\usepackage[pdftex]{hyperref}
\usepackage{enumitem}

\theoremstyle{plain}
\newtheorem{theorem}{Theorem}[section]
\newtheorem{corollary}[theorem]{Corollary}
\newtheorem{lemma}[theorem]{Lemma}
\newtheorem{proposition}[theorem]{Proposition}

\theoremstyle{definition}
\newtheorem{remark}[theorem]{Remark}
\newtheorem{definition}[theorem]{Definition}

\numberwithin{equation}{section}

\newcommand{\floor}[1]{\left\lfloor{#1}\right\rfloor}

\newcommand{\nat}{\mathbb{N}}
\newcommand{\ent}{\mathbb{Z}}

\newcommand{\vf}{\varphi}
\newcommand{\eps}{\varepsilon}
\newcommand{\abs}[1]{\left\vert #1\right\vert }

\newcommand{\bg}{\medskip\goodbreak}

\newcommand\egdef{{\,{\buildrel\rm def\over=}\,}}

\newcommand{\Fix}{\mathfrak{F}}
\newcommand{\Fn}{\mathfrak{F}^{(n)}}
\newcommand{\fix}{\mathfrak{f}}
\newcommand{\fn}{\mathfrak{f}^{(n)}}

\title[Counting Colorful Necklaces and Bracelets in Three Colors]
{{\large Counting Colorful Necklaces and Bracelets in Three Colors}}
\author[Dennis S. Bernstein]{Dennis S. Bernstein}
\address{%
Department of Aerospace Engineering\\
3020 FXB Building\\
1320 Beal St.\\
The University of Michigan\\
Ann Arbor, MI 48109-2140 }
\email{dsbaero@umich.edu}

\author[Omran Kouba]{Omran Kouba}
\address{Department of Mathematics \\
	Higher Institute for Applied Sciences and Technology\\
	P.O. Box 31983, Damascus, Syria.}
\email{omran\_kouba@hiast.edu.sy}
\thanks{This is a pre-print of an article published in Aequationes Mathematic\ae. The final authenticated version is available online at:
\href{https://doi.org/10.1007/s00010-019-00645-w}{https://doi.org/10.1007/s00010-019-00645-w.}}
\begin{document}
%\date{\today}
\begin{abstract} 
A necklace or bracelet is \textit{colorful} if no pair of adjacent beads are the same color.
In addition, two necklaces are  \textit{equivalent} if one  results from the other by permuting its colors, and two bracelets  are \textit{equivalent} if one results from the other by either permuting its colors or reversing the order of the beads;  a bracelet is thus a necklace that can be turned over. This note counts the number $K(n)$ of non-equivalent colorful necklaces and the number $K'(n)$ of colorful bracelets formed with $n$-beads in at most three colors.
Expressions obtained for $K'(n)$ simplify expressions given by OEIS sequence A114438, while the expressions given for $K(n)$ appear to be new and are not included in OEIS.

\end{abstract}
\maketitle

%%%%%%%%%%%%%%%%%%%%%%%%%%%%%%%%%%%%%%%%%%%%%%%%%%%
\section{Introduction}\label{sec0}
%%%%%%%%%%%%%%%%%%%%%%%%%%%%%%%%%%%%%%%%%%%%%%%%%%
%%%%%%%%%%%%%%%%%%%%%%%%%%%%%%%%%%%%%%%%%%%%%%%%%%%%
\bg
A {\it necklace} with $n$ beads and $c$ colors is an $n$-tuple, each of whose components can assume one of $c$ values, where not all $c$ colors need appear.  
Two necklaces are \textit{equivalent} if one results from the other by either rotating it cyclically or permuting its colors.
The classical necklace problem asks to determine the number of non-equivalent necklaces formed with $n$ beads of $c$ colors.
The answer to this problem is given by 
\begin{equation}\label{eq10}
N(n,c)=\frac{1}{n}\sum_{d|n}\vf(d)c^{n/d},
\end{equation} 
where $\vf$ is the Euler totient function. 

A {\it bracelet} with $n$ beads and $c$ colors is a necklace of $n$ beads and $c$ colors that can be turned over, and thus the order of its beads is reversed.
The number of non-equivalent bracelets with $n$ beads of $c$ colors is given by 
\begin{equation}
N'(n,c)=\frac{N(n,c)+R(n,c)}{2},
\end{equation}
where
\begin{equation}
R(n,c)=
\begin{cases} c^{(n+1)/2}&\textrm{if $n$ is odd},\\
\frac{1+c}{2} c^{n/2}&\textrm{if $n$ is even}.
\end{cases}
\end{equation}
As expected, $N'(n,c)\le N(n,c)$.
These results and further details are given in \cite{wei,wei1} and the references therein.

In this work we consider a variation on this problem that arose from considering sequences of $n$ coordinate-axis rotations defined by Euler angles, where $n=7$ for aircraft \cite{Ber}.
For this problem, it is of interest to count the number of distinct coordinate-axis rotation sequences of length $n$ that are closed in the sense of transforming the starting frame by a sequence of coordinate-axis rotations that lead back to the starting frame \cite{bhatcrasta}.
A pair of successive coordinate-axis rotations around the same axis can be combined into a single rotation, and the labeling of the axes of the starting frame is arbitrary.
Counting the number of closed sequences consisting of $n$ coordinate-axis rotations is thus equivalent to counting necklaces in $3$ colors, where each color corresponds to an axis label.
Furthermore, reversing a sequence of coordinate-axis rotations is equivalent to replacing each Euler angle in the sequence of coordinate-axis rotations by its negative.
Hence, for the purpose of determining all feasible Euler angles for each closed sequence of coordinate-axis rotations, it suffices to count bracelets.

Motivated by the fact that successive coordinate-axis rotations around the same axis can be merged, the present paper considers colorful necklaces and bracelets formed with $n$-beads of three colors, where a necklace or bracelet is \textit{colorful} if no pair of adjacent beads have the same color. 
Two colorful necklaces are \textit{equivalent} if one results from the other by  permuting its colors, and two colorful bracelets 
are \textit{equivalent} if one results from the other by either permuting its colors or reversing the order of the beads; a colorful bracelet is thus a colorful necklace that can be turned over.
In fact, the number of colorful bracelets with $n$ beads in three colors appears in the
On-line Encyclopedia of Integer Sequences (OEIS)
as sequence \href{https://oeis.org/A114438}{A114438}, which is the ``Number of Barlow packings that repeat after $n$  (or a divisor of $n$) layers.'' 
The provided references indicate that the problem of studying this sequence originates in crystallography \cite{McL,ram,Th}.

The contribution of the present paper is twofold.
First, we provide explicit formulas for the number of color bracelets that simplify those given by $P'(n)$ in \cite[p. 272]{McL}. 
Furthermore, we provide expressions for the number of colorful necklaces with $n$ beads in three colors; this sequence is currently unknown to OEIS.

%%%%%%%%%%%%%%%%%%%%%%%%%%%%%%%%%%%%%%%%%%%%%%%%%%%
\section{{n}-Periodic Sequences}\label{sec1}
%%%%%%%%%%%%%%%%%%%%%%%%%%%%%%%%%%%%%%%%%%%%%%%%%%
%%%%%%%%%%%%%%%%%%%%%%%%%%%%%%%%%%%%%%%%%%%%%%%%%%%%

Instead of working with necklaces and bracelets of $n$ beads we work with $n$-periodic sequences.
To fix notation, let $\nat_q$ denote the set of natural numbers from $1$ to $q$. In particular, $\nat_3$ represents the set of the three colors under consideration. 
	
\begin{definition}
For a positive integer $n$,
a \textit{colorful} $n$-periodic sequence
is a function $f:\ent/n\ent\to\nat_3$ defined on the set of integers modulo $n$, such that  $f(i)\ne f(i+1)$ for all $i\in\ent/n\ent$. 
The set of colorful $n$-periodic sequences is denoted by $\mathcal{A}_n$. 
\end{definition} 

The set $\mathcal{A}_n$ represents the colorful $n$-bead necklaces or bracelets of three colors.

Next, we consider the permutations $r$ and $s$ on $\ent/n\ent$ defined by  $s(i)=i+1$ and $r(i)=-i$ for all $i\in\ent/n\ent$. The  group generated by $s$ is  
\begin{equation*} 
\langle s\rangle =\{id,s,s^2,\ldots,s^{n-1}\},
\end{equation*}
which is isomorphic to $\ent/n\ent$. The group generated by $r$ is $\langle r\rangle=\{id,r\}\simeq\ent/2\ent$. Furthermore, since $rsr=s^{-1}$, the group generated by $r$ and $s$ is 
\[ \langle s,r\rangle =\{id,s,s^2,\ldots,s^{n-1},r,rs,rs^2,\ldots,rs^{n-1}\},\]
which is isomorphic to the dihedral group $D_n$ of order $2n$.

We also consider the symmetric group $\mathfrak{S}_3$ of $\nat_3$, where  $\mathfrak{S}_3$ consists of the identity $id$, the three substitutions
$\tau_{12},$ $\tau_{13},$ and $\tau_{23}$ (with $\tau_{ij}$ representing the $2$-cycle $(i,j)$), and the two $3$-cycles $c=(1,2,3)$ and $c^2=c^{-1}=(1,3,2)$.
We recall that two permutations $\sigma$ and $\sigma'$ are \textit{conjugates} if there exists a permutation $\tau$ such that $\sigma'=\tau^{-1}\circ\sigma\circ\tau$; this is equivalent to the fact that $\sigma$ and $\sigma'$ have the same cyclic decomposition.  Hence, for $\mathfrak{S}_3$ all transpositions are conjugates and all $3$-cycles are conjugates, which can be checked directly.

The group acting on the elements of $\mathcal{A}_n$, considered as  colorful $n$-bead necklaces, is $G=\mathfrak{S}_3\times \langle s\rangle$, with group action  defined by $\gamma f=\sigma\circ f\circ t^{-1}$ for 
$\gamma=(\sigma,t)\in G$ and $f\in\mathcal{A}_n$. Similarly, the group acting on the elements of $\mathcal{A}_n$, considered as  colorful $n$-bead bracelets, is $G'=\mathfrak{S}_3\times \langle s,r\rangle$, with group action  defined again by $\gamma f=\sigma\circ f\circ t^{-1}$ for 
$\gamma=(\sigma,t)\in G'$ and $f\in\mathcal{A}_n$.

The action of $G$ on $\mathcal{A}_n$ defines an equivalence relation $\sim_G$ on colorful $n$-bead necklaces given by
\[f\sim_Gg\iff \mbox{there exists } \gamma\in G \mbox{ such that } \gamma f=g.\]
The set of equivalence classes of this relation is denoted by
$\mathcal{A}_n/G$, and the desired number of non-equivalent colorful $n$-bead necklaces is exactly  $K(n)=\abs{\mathcal{A}_n/G}$. Similarly, $K'(n)=\abs{\mathcal{A}_n/G'}$, where $\mathcal{A}_n/G'$ is the set of equivalence classes of colorful $n$-bead bracelets defined by the action of $G'$ on  $\mathcal{A}_n$.

The basic tool in this investigation is the classical Burnside's Lemma \cite[Theorem 3.22]{rot}.  (While this lemma bears the name of Burnside, it seems that it was well-known to Frobenius (1887) and before him to Cauchy (1845). An account on the history of this lemma can be found in \cite{wei}, see also \cite{wrt}).

\begin{theorem}[Burnside's Lemma] \label{thB} Let $G$ be a finite group  acting on a finite set $X$. Then
	\[\abs{X/G}=\frac{1}{\abs{G}}\sum_{\gamma\in G}\abs{\Fix(\gamma)},\]
where $\Fix(\gamma)$ is the set of elements $x\in X$ fixed by $\gamma$ (\textit{i.e.} $\gamma x=x$.)
\end{theorem}

\bg
Noting that $G$ is a subgroup of $G'$, our task is to determine the numbers
\begin{equation}
	\fn(\sigma,\eps,j)=\abs{\Fn(\sigma,\eps,j)}
\end{equation}
 with $\sigma\in\mathfrak{S}_3$, $\eps\in\{0,1\},$ and $j\in
\{0,\ldots,n-1\}$, where
\begin{equation}
\Fn(\sigma,\eps,j)=\left\{f\in\mathcal{A}_n:f=\sigma\circ f\circ r^\eps\circ s^j\right\}. 
\end{equation}

This paper is organized as follows: In Section \ref{sec2}, we gather some useful properties and lemmas. In Section \ref{sec3} the case of necklaces is considered. Finally, in Section \ref{sec4} we consider the case of bracelets.

%%%%%%%%%%%%%%%%%%%%%%%%%%%%%%%%%%%%%%%%%%%%%%%%%%%
\section{Useful Properties and Lemmas}\label{sec2}
%%%%%%%%%%%%%%%%%%%%%%%%%%%%%%%%%%%%%%%%%%%%%%%%%%%
%%%%%%%%%%%%%%%%%%%%%%%%%%%%%%%%%%%%%%%%%%%%%%%%%%%
%The following lemma is a straightforward consequence of the definitions.

\begin{lemma}\label{lm1}
If $n$ and $m$ are positive integers, then
$\mathcal{A}_n\cap\mathcal{A}_m=\mathcal{A}_{\gcd(n,m)}.$
\end{lemma} 
\begin{proof}
This result follows from the fact that $n\ent+m\ent=\gcd(n,m)\ent$.
\end{proof}

Our first step is to determine the number of $n$-periodic colorful sequences, that is $\alpha_n\egdef\abs{\mathcal{A}_n}$. This is the object of the next proposition.
 
\begin{proposition}\label{pr2}
For all $n\ge1$, 
\begin{equation}\label{eq11}
\alpha_n=\abs{\mathcal{A}_n}=2^n+2(-1)^n.
\end{equation}
\end{proposition}

\begin{proof}
 Note that $\alpha_1=0$ and $\alpha_2=6$. Suppose that $n\ge3$, and define
\begin{align*}
\mathcal{A}_n^{\prime}&=\{f\in\mathcal{A}_n:f(n-1)\ne f(1)\},\\
\mathcal{A}_n^{\prime\prime}&=\{f\in\mathcal{A}_n:f(n-1)= f(1)\}.
\end{align*}

The mapping $\mathcal{A}_n^{\prime}\to\mathcal{A}_{n-1}:
f\mapsto \tilde{f}$ with $\tilde{f}$ defined by
$\tilde{f}_{|\nat_{n-1}}=f_{|\nat_{n-1}}$ is bijective because $f(n)$ is uniquely defined by the knowledge of $f(n-1)$ and $f(1)$, indeed $\{f(1),f(n-1),f(n)\}=\nat_3$.
Hence $\abs{\mathcal{A}_n^{\prime}}=\abs{\mathcal{A}_{n-1}}=\alpha_{n-1}$.

Also, the mapping $\mathcal{A}_n^{\prime\prime}\to\mathcal{A}_{n-2}:
f\mapsto \hat{f}$ with $\hat{f}$ defined by
$\hat{f}_{|\nat_{n-2}}=f_{|\nat_{n-2}}$ is surjective
and the pre-image of each $g\in\mathcal{A}_{n-2}$
consists of exactly two elements, namely
$f_1$ and $f_2$ defined by
$f_{1|\nat_{n-2}}=f_{2|\nat_{n-2}}=g_{|\nat_{n-2}}$,
$f_1(n)=\min(\nat_3\setminus\{f(1)\})$ and
$f_2(n)=\max(\nat_3\setminus\{f(1)\})$. 
Hence $\abs{\mathcal{A}_n^{\prime\prime}}=2\abs{\mathcal{A}_{n-2}}
=2\alpha_{n-2}$. But $\left\{\mathcal{A}_n^{\prime},
\mathcal{A}_n^{\prime\prime}\right\}$ is a partition of $\mathcal{A}_n$,
so
\[
\alpha_n=\abs{\mathcal{A}_n}=
\abs{\mathcal{A}_n^{\prime}}+
\abs{\mathcal{A}_n^{\prime\prime}}=\alpha_{n-1}+2\alpha_{n-2},
\]
and the desired conclusion follows by  induction.
\end{proof}

In particular, since the neutral element of $G$ (or $G'$) fixes the whole set $\mathcal{A}_n$, the next corollary is immediate.
\begin{corollary} \label{cor1}
 \begin{equation}
 	\fn(id,0,0)=\alpha_n.
 \end{equation}
\end{corollary} 

\begin{corollary}\label{cor2}
For distinct $i,j\in\nat_3$, let $\mathcal{A}_n^{i\cdot\cdot j}$ denote
the subset of $\mathcal{A}_n$ consisting of functions $f$ satisfying  $f(1)=i$ and $f(n)= j$.  Then
\begin{equation}
\abs{\mathcal{A}_n^{i\cdot\cdot j}}=\frac{\alpha_n}{6}.
\end{equation}
\end{corollary}
\begin{proof}
Given $i$ and $j$, there is a unique permutation $\sigma\in\mathfrak{S}_3$ such that $\sigma(i)=1$ and $\sigma(j)=2$, and with this $\sigma$ the mapping $f\mapsto \sigma\circ f$ defines a bijection between
$\mathcal{A}_n^{i\cdot\cdot j}$ and $\mathcal{A}_n^{1\cdot\cdot 2}$. Thus
\[\abs{\mathcal{A}_n^{i\cdot\cdot j}}=\abs{\mathcal{A}_n^{1\cdot\cdot 2}}.\]
The conclusion follows since $\{\mathcal{A}_n^{1\cdot\cdot 2},\mathcal{A}_n^{1\cdot\cdot 3},\mathcal{A}_n^{2\cdot\cdot 1}
,\mathcal{A}_n^{2\cdot\cdot 3},\mathcal{A}_n^{3\cdot\cdot 1},\mathcal{A}_n^{3\cdot\cdot 2} \}$
constitutes a partition of  $\mathcal{A}_n$. 
\end{proof}

The next lemma helps to reduce the number of cases to be considered.  The proof is immediate and left to the reader.

\begin{lemma}[Reduction]\label{lm3}
Suppose that a group $G$ acts on a set $X$, and consider two elements $g$ and $g'$ from $G$. If there is $h\in G$ such that
$g'=h^{-1}gh$, then the mapping $x\mapsto hx$ defines a bijection from $\Fix(g')$ onto $\Fix(g)$. In particular, if $X$ and $G$ are finite and if $g$ and $g'$ are conjugate elements from $G$ then $\abs{\Fix(g)}=\abs{\Fix(g')}$.
\end{lemma}

\begin{remark}\label{rm1}
Another simple remark from group theory is that if $G=A\times B$ is the direct product of two groups $A$ and $B$, and if $a$ and $a'$ are conjugate elements from $A$, then $(a,e_B)$ and $(a',e_B)$, (with $e_B$ denoting the neutral element of $B$), are also conjugate elements in $G$. 
\end{remark}
 
 With Lemma \ref{lm3} and Remark \ref{rm1} at hand, the next corollary is immediate:
\bigskip\nobreak
\begin{corollary} \label{cor27}$\,$
\begin{enumerate}[label=\textbf{\upshape{(\alph*)}}]
\item  For all $\eps\in\{0,1\}$ and all $\ell\in\{0,1,\ldots,n-1\}$ we have 
\begin{equation}
\fn(\tau_{12},\eps,\ell) =
\fn(\tau_{13},\eps,\ell)=
\fn(\tau_{23},\eps,\ell).
\end{equation}
\item  For all $\eps\in\{0,1\}$ and all $\ell\in\{0,1,\ldots,n-1\}$ we have 
\begin{equation}
\fn(c,\eps,\ell) =
\fn(c^2,\eps,\ell)
\end{equation}
\item  For all $\sigma\in\mathfrak{S}_3$ and all $\ell\in\{0,1,\ldots,\floor{(n-1)/2}\}$ we have 
\begin{equation}
\fn(\sigma,1,2\ell)=\fn(\sigma,1,0)
\end{equation}
\item  For all $\sigma\in\mathfrak{S}_3$ and all $\ell\in\{0,1,\ldots,\floor{n/2-1}\}$ we have 
\begin{equation}
\fn(\sigma,1,2\ell+1) =\fn(\sigma,1,1)
\end{equation}
\end{enumerate}
\end{corollary}
\begin{proof}
Both \textbf{(a)} and \textbf{(b)} follow from the fact that all permutations of the same cycle structure are conjugate.
On the other hand, since $rs^{2\ell}=s^{-\ell}rs^{\ell}$ and 
$rs^{2\ell+1}=s^{-\ell}(rs)s^{\ell}$ for all $\ell$, both \textbf{(c)} and \textbf{(d)} follow from Corollary \ref{cor27}.
\end{proof}
The final result in this preliminary section is a simple formula concerning sums involving Euler's totient function $\vf$ (see \cite[Chapter V, Section 5.5]{har}), recall that $\vf(n)$ is the number of integers in $\nat_n$  coprime to $n$.
\begin{lemma}\label{lm2}
For every positive integer $n$ we have
\[\sum_{d|n}(-1)^d\vf\left(\frac{n}{d}\right)=\begin{cases}0&\textrm{if $n$ is even,}\\
-n&\textrm{if $n$ is odd.} \end{cases}\] 
\end{lemma}
\begin{proof}
If $n$ is odd then all its divisors are odd and using \cite[Theorem 63]{har}, we get
\[
\sum_{d|n}(-1)^d\vf\left(\frac{n}{d}\right)=-
\sum_{d|n}\vf\left(\frac{n}{d}\right)=-n.
\]

Now, if $n=2m$ for some positive integer $m$, then
\begin{align*}
	\sum_{d|n}(-1)^d\vf\left(\frac{n}{d}\right)&=
	\sum_{\genfrac{}{}{0pt}{1}{d|n}{ d \textrm{ is even}}} \vf\left(\frac{n}{d}\right)-\sum_{\genfrac{}{}{0pt}{1}{d|n}{ d \textrm{ is odd}}} \vf\left(\frac{n}{d}\right)\\
	&=2\sum_{\genfrac{}{}{0pt}{1}{d|n}{ d \textrm{ is even}}} \vf\left(\frac{n}{d}\right)-\sum_{d|n } \vf\left(\frac{n}{d}\right)\\
	&=2\sum_{ d'|m} \vf\left(\frac{m}{d'}\right)-\sum_{d|n } \vf\left(\frac{n}{d}\right)\\
	&=2m-n=0,
\end{align*}
where we used again \cite[Theorem 63]{har}. 
\end{proof}

%%%%%%%%%%%%%%%%%%%%%%%%%%%%%%%%%%%%%%%%
%%%%%%%%%%%%%%%%%%%%%%%%%%%%%%%%%%%%%%%%
\section{Counting Colorful Necklaces}\label{sec3}
%%%%%%%%%%%%%%%%%%%%%%%%%%%%%%%%%%%%%%%%

In this section we consider $G=\mathfrak{S}_3\times\langle s\rangle$. According to Corollary \ref{cor27} we need to determine $\fn(\sigma,0,\ell)$, for $\sigma\in\{id,\tau_{12},c\}$ and $\ell\in\ent/n\ent$.
The next proposition gives the answer.

\begin{proposition}\label{p31}$\,$
\begin{enumerate}[label=\textbf{\upshape{(\alph*)}}]
\item  
If $\gcd(\ell,n)=d$ then $\Fn(id,0,\ell)=\Fn(id,0,d)=\mathcal{A}_d$. In particular,
\begin{equation}
\fn(id,0,\ell)=\alpha_{\gcd(\ell,n)}.
\end{equation}
Thus, $\gcd(n,\ell)=1$ implies $\Fn(id,0,\ell)=\emptyset$.
\item 
\begin{enumerate}[label=\textit{\roman*.},leftmargin=0.3cm]
\item 
Suppose that $3\nmid n$, then for  all $\ell\in\ent/n\ent$ we have \begin{equation}\fn(c,0,\ell)=0.\end{equation}  
\item Suppose that $n=3m$ for some positive integer $m$, then for  all $\ell\in\ent/n\ent$ we have
\begin{equation}
\fn(c,0,\ell) =
 \begin{cases}
          0         &\textrm{if $3\mid  (\ell/d)$,}\\
          2^d-(-1)^d&\textrm{if $3\nmid (\ell/d)$,}
 \end{cases}
 \end{equation}
  where $d=\gcd(m,\ell)$. 
  \end{enumerate} 
\item 
\begin{enumerate}[label=\textit{\roman*.},leftmargin=0.3cm]
\item 
Suppose that $2\nmid n$, then for  all $\ell\in\ent/n\ent$ we have 
\begin{equation}
     \fn(\tau_{12},0,\ell)=0.
\end{equation}
\item  Suppose that $n=2m$ for some positive integer $m$, then  for all $\ell\in\ent/n\ent$ 
  we have
\begin{equation}
 \fn(\tau_{12},0,\ell) =
 \begin{cases}
        0 &\textrm{if $2\mid (\ell/d)$,}\\
      2^d &\textrm{if $2\nmid (\ell/d)$,}
 \end{cases}
\end{equation} 
where $d=\gcd(m,\ell)$. 
\end{enumerate}
\end{enumerate}
\end{proposition}
\begin{proof}
\textbf{(a)} A sequence $f\in \Fn((id,0,\ell))$ satisfies $f\circ s^\ell=f$, so it belongs to $\mathcal{A}_\ell$. Thus, by Lemma \ref{lm1}, we have
\[\Fn((id,0,\ell))\subset\mathcal{A}_n\cap \mathcal{A}_\ell=\mathcal{A}_d.\]
The converse inclusion: $\mathcal{A}_d\subset\Fn((id,0,\ell))$ is trivial, because both $\ell$ and $n$ are multiples of $d$. 
%%%%%%%%%%%%%%%%%%%%%%%%%%%%%%%%%%%

%%%%%%%%%%%%%%%%%%%%%%%%%%%%%%%%%%%
\noindent\textbf{(b,c)} \textit{i.} Let $\sigma$ be any permutation from $\mathfrak{S}_3$, and
suppose that $\Fn(\sigma,0,\ell)\ne\emptyset$ so there is  $f\in\Fn(\sigma,0,\ell)$. From 
\begin{equation}\label{eq2101}
f\circ s^\ell=\sigma^{-1}\circ f,
\end{equation} 
we conclude by an easy induction that for all integers $p$ we have
\begin{equation}\label{eq2102}
 f\circ s^{p\ell}=\sigma^{-p}\circ f
\end{equation} 
\begin{itemize}[leftmargin=0.5cm]
\item If $\sigma=c$ and $3\nmid n$ we have $c^3=id$, so \eqref{eq2102} implies that
$f\circ s^{3p\ell}= f$ for all integers $p$. But, because $3\nmid n$  there is $r\in\{1,2\}$ such that $n-r=3p$ for some $p$. Consequently, $f\circ s^{(n-r)\ell}=f$, or equivalently
$f=f\circ s^{r\ell}=c^{-r}\circ f$ because $f$ is $n$-periodic. 
This is a contradiction because neither $c$ nor $c^2$ has fixed points. Thus $\Fn(c,0,\ell)=\emptyset$. This proves \textbf{(b)} $i$.
\item 
If $\sigma=\tau_{12}$ and $n$ is odd, we have $\tau_{12}^2=id$, so \eqref{eq2102} implies that
$f\circ s^{2p\ell}= f$ for all integers $p$. But, because $n=2p+1$ for some $p$ we conclude that $f\circ s^{(n-1)\ell}=f$, or equivalently
$f=f\circ s^{\ell}=\tau_{12}\circ f$. 
This is a contradiction because $f$ takes two different values, and $\tau_{12}$ has only one fixed point. Thus $\Fn(\tau_{12},0,\ell)=\emptyset$. This proves \textbf{(c)} $i$. 
\end{itemize}
%%%%%%%%%%%%%%%%%%%%%%%%%%%%%%%%%%%

%%%%%%%%%%%%%%%%%%%%%%%%%%%%%%%%%%%
\noindent\textbf{(b)} $ii$. Assume that $\Fn(c,0,\ell)\ne\emptyset$ and consider $f\in \Fn(c,0,\ell)$. From $c\circ f\circ s^\ell= f$ we conclude that $f\circ s^{3\ell}=f$. Thus  $f\in\mathcal{A}_{3j}\cap\mathcal{A}_{3m}=\mathcal{A}_{3d}$, with $d=\gcd(m,\ell)$. 
Further, if $\ell/d=3q+r$ with $r\in\{0,1,2\}$ then $
\ell=3dq+dr$ and consequently 
\begin{equation}\label{eq2111}
f\circ s^\ell=f\circ s^{rd}=c^2\circ f.
\end{equation}
\begin{itemize}[leftmargin=0.5cm]
\item If $r=0$ then \eqref{eq2111} implies $f=c^2\circ f$ which is impossible since $f$ is not constant. Thus $\Fn(c,0,\ell)=\emptyset$ in this case.
\item If $r=1$ then  \eqref{eq2111} shows that $f\in\Fix^{(3d)}(c,0,d)$. Conversely, it is easy to check that any $f\in  \Fix^{(3d)}(c,0,d)$ belongs to $\Fn(c,0,\ell)$. Thus, we have shown that
\begin{equation}\label{eq2112}
\Fn(c,0,\ell)= \Fix^{(3d)}(c,0,d).
\end{equation}
Now, when $f$ belongs to $\Fix^{(3d)}(c,0,d)$ it is completely determined by its restriction to $\nat_d$,  and 
the mapping (see Figure \ref{fig1}):
\[\Phi:\mathcal{A}_{d+1}^{1\cdot\cdot3}\cup 
\mathcal{A}_{d+1}^{2\cdot\cdot1}\cup \mathcal{A}_{d+1}^{3\cdot\cdot2}\to \Fix^{(3d)}(c,0,d),f\mapsto \tilde{f}, \]
 where $\tilde{f}$ is the unique  sequence 
 from $ \Fix^{(3d)}(c,0,d)$ which  coincides with
$f$ on $\nat_d$, is a bijection. 
\begin{figure}[h!]
\begin{center}
$\boxed{\displaystyle
\underbrace{(x_1,\ldots,x_d)}_{f_{|\nat_d}}\mapsto
\underbrace{(x_1,\ldots,x_d,c^2(x_1),\ldots,c^2(x_d),
c(x_1),\ldots,c(x_d))}_{\tilde{f}_{|\nat_{3d}}}}
$
\end{center}
\caption{The bijection $\Phi:\mathcal{A}_{d+1}^{1\cdot\cdot3}\cup 
\mathcal{A}_{d+1}^{2\cdot\cdot1}\cup \mathcal{A}_{d+1}^{3\cdot\cdot2}\to \Fix^{(3d)}(c,0,d)$}
\label{fig1}
\end{figure}

\noindent We conclude, according to Corollary \ref{cor2} that 
\[ \fix^{(3d)}(c,0,d)=\abs{\mathcal{A}_{d+1}^{1\cdot\cdot3}}+\abs{\mathcal{A}_{d+1}^{2\cdot\cdot1}}+\abs{\mathcal{A}_{d+1}^{3\cdot\cdot2}}=\frac{\alpha_{d+1}}{2}.\] 
\item If $r=2$, then a similar argument to  the previous one  (with $c$ replaced by $c^2$), yields the desired conclusion. This completes the proof of \textbf{(b)} $ii$.
\end{itemize}
%%%%%%%%%%%%%%%%%%%%%%%%%%%%%%%%%%%

%%%%%%%%%%%%%%%%%%%%%%%%%%%%%%%%%%%
\noindent\textbf{(c)} $ii$. For simplicity we  write $\tau$ for $\tau_{12}$.
Suppose that $\Fn(\tau,0,\ell)\ne\emptyset$ and consider $f\in\Fn(\tau,0,\ell)$. We have
\begin{equation}\label{eq121}
f\circ s^\ell=\tau\circ f.
\end{equation}
Hence $f\circ s^{2\ell}=f$ and consequently $f\in\mathcal{A}_{2\ell}$, this implies that $f\in
\mathcal{A}_{2m}\cap\mathcal{A}_{2\ell}=\mathcal{A}_{2d}$ where $d=\gcd(m,\ell)$.  Now write $\ell/d=2q+r$ with $r\in\{0,1\}$, then $\ell=2dq+dr$ and consequently 
\begin{equation}\label{eq2121}
f\circ s^\ell=f\circ s^{rd}=\tau\circ f.
\end{equation}

\begin{itemize}[leftmargin=0.5cm]
\item If $r=0$, then \eqref{eq2121} implies $f=\tau\circ f$ which is impossible since $f$ is not constant. Thus $\Fn(\tau,0,\ell)=\emptyset$ in this case.
\item If $r=1$, then \eqref{eq2121} implies $f\circ s^d=\tau\circ f$, that is $f\in\Fix^{(2d)}(\tau,0,d)$. Conversely, it is easy to check that any $f\in  \Fix^{(2d)}(\tau,0,d)$ belongs to $\Fn(\tau,0,\ell)$. Thus, we have shown that in this case
\begin{equation}\label{eq2123}
\Fn(\tau,0,\ell)= \Fix^{(2d)}(\tau,0,d).
\end{equation}
Clearly, if $d=1$ then $\Fix^{(2d)}(\tau,0,d)$ consists exactly of two elements: namely $f_1$, defined by $f_{1}(1)=1,f_1(2)=2$, and  $f_2=\tau\circ f_1$. So,
\[{\fix^{(2)}(\tau,0,1))=2}.\]
Now suppose that $d>1$. Any $f\in \Fix^{(2d)}((\tau,d))$ is completely determined its restriction to $\nat_d$ (note that $f(d)$ should be different from $\tau(f(1))$ and $f(d-1)$,) so considering the different possibilities for $f(1)$ we see that
the mapping $\Psi$, (see Figure \ref{fig2}):

\[\Psi:\mathcal{A}_{d+1}^{1\cdot\cdot2}\cup 
\mathcal{A}_{d+1}^{2\cdot\cdot1}\cup \mathcal{A}_{d}^{3\cdot\cdot1}\cup \mathcal{A}_{d}^{3\cdot\cdot2}\to \Fix^{(2d)}(\tau,0,d),f\mapsto\hat{f},
\] where  
$\hat{f}$ is the unique sequence from $ \Fix^{(2d)}(\tau,0,d)$ that  coincides with $f$ on $\nat_d$, is a bijection. 

\begin{figure}[h!]
	\begin{center}
		$\boxed{\displaystyle
			\underbrace{(x_1,\ldots,x_d)}_{f_{|\nat_d}}\mapsto
			\underbrace{(x_1,\ldots,x_d,\tau(x_1),\ldots,\tau(x_d))}_{\hat{f}_{|\nat_{2d}}}}
		$
	\end{center}
	\caption{The bijection $\Psi:\mathcal{A}_{d+1}^{1\cdot\cdot2}\cup 
		\mathcal{A}_{d+1}^{2\cdot\cdot1}\cup \mathcal{A}_{d}^{3\cdot\cdot1}\cup \mathcal{A}_{d}^{3\cdot\cdot2}\to \Fix^{(2d)}(\tau,0,d)$}
	\label{fig2}
\end{figure}

\noindent Thus, according to Corollary \ref{cor2}, we have
\[\fix^{(2d)}(\tau,0,d)=\frac{\alpha_{d+1}+\alpha_d}{3}=2^d.\]
\end{itemize}
This concludes the proof of  \textbf{(c)} $ii$, in view of \eqref{eq2123}.
\end{proof}

The final step is to put all the pieces together to get the expression of $K_n$ in terms of $n$ using Burnside's Lemma.

\begin{theorem}\label{th214}
The number of non-equivalent colorful $n$-bead necklaces with three colors is given by 
\begin{align}
K(n)&=\frac{1}{6n}\sum_{d|n}
(1+\mathbb{I}_{2\ent\setminus3\ent}(d))\gcd(d,6)\,\vf(d) 2^{n/d}-\frac{1}{3^{1+\nu_3(n)}} \mathbb{I}_{\ent\setminus2\ent}(n)\label{eqexp1}\\
&=\floor{\frac{1}{6n}\sum_{d|n}
(1+\mathbb{I}_{2\ent\setminus3\ent}(d))\gcd(d,6)\,\vf(d) 2^{n/d}}\label{eqexp2}
\end{align}
where $\mathbb{I}_X$ is the indicator function of the set $X$, ( \textit{i.e.} $\mathbb{I}_X(k)=1$ if $k\in X$ and $\mathbb{I}_X(k)=0$ if $k\notin X$,) and $3^{\nu_3(n)}$ is the largest power of $3$ dividing $n$.
\end{theorem} 
\begin{proof}
According to Corollary \ref{cor27} and Burnside's Lemma \ref{thB} we have
\begin{equation}\label{eq210}
K(n)=\frac{A_n+3B_n+2C_n}{6n}
\end{equation} 
with
\begin{align}
A_n&=\sum_{\ell=0}^{n-1}\fn(id,0,\ell),\\
B_n&=\sum_{\ell=0}^{n-1}\fn(\tau_{12},0,\ell),\\
C_n&=\sum_{\ell=0}^{n-1}\fn(c,0,\ell).
\end{align}
Using part \textbf{(a)} of Proposition \ref{p31} we have
\begin{align*}
A_n&=
\sum_{\ell=0}^{n-1}\alpha_{\gcd(\ell,n)}=
\sum_{d|n} \alpha_{d}\abs{\{\ell:0\le\ell<n, \gcd(\ell,n)=d\}}\\
&=
\sum_{d|n} \alpha_{d}\abs{\left\{\ell':0\le\ell'<\frac{n}{d}, \gcd\left(\ell',\frac{n}{d}\right)=1\right\}}
=\sum_{d|n}\vf\left(\frac{n}{d}\right) \alpha_{d}.
\end{align*}
Thus,  using the expression of $\alpha_n$ from Proposition \ref{pr2}, we get
\begin{equation}\label{eq214}
A_n=\sum_{d|n}\left(2^d+2(-1)^d\right) \vf\left(\frac{n}{d}\right).
\end{equation}
Similarly, according part \textbf{(c)} of Proposition \ref{p31} we know that $B_n=0$ if $n$ is odd, while we have the following when $n=2m$:
\begin{align*}
B_n&=\sum_{\ell=0}^{n-1}2^{\gcd(\ell,m)}
\mathbb{I}_{\ent\setminus2\ent}\left(\frac{\ell}{\gcd(\ell,m)}\right)\\
&=\sum_{d|m}2^d\abs{\left\{0\le\ell<2m:\gcd(\ell,m)=d,
\textrm{and~}\ell/d \textrm{ is odd}\right\}
}\\
&=\sum_{d|m}2^d\abs{\left\{0\le\ell'<2\frac{m}{d}:\gcd(\ell',\frac{m}{d})=1,
\textrm{and~}\ell' \textrm{ is odd}\right\}
}\\
&=\sum_{d|m}2^d\abs{\left\{0\le\ell'<2\frac{m}{d}:\gcd(\ell',2\frac{m}{d})=1\right\}
}=\sum_{d|m}2^d\vf\left(\frac{n}{d}\right).
\end{align*}

 Finally we get
\begin{equation}\label{eq215}
B_n=\mathbb{I}_{2\ent}(n) \sum_{d|(n/2)}2^d\vf\left(\frac{n}{d}\right)
\end{equation}

Now, we come to $C_n$. According to part \textbf{(b)} of Proposition \ref{p31} we know that $C_n=0$ if $n$ is not a multiple of $3$ while if $n=3m$ we have
\begin{align*}
C_n&=\sum_{\ell=0}^{n-1}\left(2^{\gcd(\ell,m)}-(-1)^{\gcd(\ell,m)}\right)\mathbb{I}_{\ent\setminus3\ent}\left(\frac{\ell}{\gcd(\ell,m)}\right)\\
&=\sum_{d|m} \left(2^{d}-(-1)^{d}\right)\abs{\left\{0\le\ell<3m:\gcd(\ell,m)=d,
\textrm{and~}3\nmid \ell/d\right\}
}\\
&=\sum_{d|m}\left(2^{d}-(-1)^{d}\right)
\abs{\left\{0\le\ell'<3\frac{m}{d}:\gcd(\ell',\frac{m}{d})=1,
\textrm{and~}3\nmid \ell'\right\}
}\\
&=\sum_{d|m}\left(2^{d}-(-1)^{d}\right)\abs{\left\{0\le\ell'<3\frac{m}{d}:\gcd(\ell',3\frac{m}{d})=1\right\}
}\\
&=\sum_{d|m}\left(2^{d}-(-1)^{d}\right)\vf\left(\frac{n}{d}\right).
\end{align*}

Thus,
\begin{equation}\label{eq216}
C_n=\mathbb{I}_{3\ent}(n) \sum_{d|(n/3)}\left(2^{d}-(-1)^{d}\right)\vf\left(\frac{n}{d}\right)
\end{equation}
Replacing \eqref{eq214},  \eqref{eq215}  and
 \eqref{eq216} in \eqref{eq210} we get
\begin{equation}
	 K(n)=b_n+\eps_n,
\end{equation}
with
\begin{equation}\label{eq218}
	b_n=\frac{1}{6n}\left(\sum_{d|n}2^d\vf\left(\frac{n}{d}\right)+
	3\mathbb{I}_{2\ent}(n)\sum_{d|(n/2)}2^d\vf\left(\frac{n}{d}\right)
	+2\mathbb{I}_{3\ent}(n)\sum_{d|(n/3)}2^d\vf\left(\frac{n}{d}\right)
	\right)
\end{equation}
and
\begin{equation}
\eps_n=\frac{1}{3n}\left(\sum_{d|n}(-1)^d\vf\left(\frac{n}{d}\right)-
\mathbb{I}_{3\ent}(n)\sum_{d|(n/3)}(-1)^d\vf\left(\frac{n}{d}\right)\right)
\end{equation}

In order to reduce a little bit the expression of $K(n)n$ we  use Lemma \ref{lm2}. Indeed, Suppose that
$n=3^\nu m$ where $\nu=\nu_3(n)$ is the exponent of $3$ in the prime factorization of $n$, thus $3\nmid m$.
Clearly if $\nu=0$ then using  Lemma \ref{lm2}  we get
\begin{equation}\label{eq219}
\eps_n=\frac{1}{3n} \sum_{d|n}(-1)^d\vf\left(\frac{n}{d}\right)=
-\frac{1}{3}\mathbb{I}_{\ent\setminus2\ent}(n)
\end{equation}
Now if $\nu>0$ then
\begin{align*} 
\eps_n&=\frac{1}{3n}\left(\sum_{d|(3^\nu m)}(-1)^d\vf\left(\frac{n}{d}\right)- \sum_{d|(3^{\nu-1}m)}(-1)^d\vf\left(\frac{n}{d}\right)\right)\\
&=\frac{1}{3n} \sum_{d|(3^\nu m),\, d\nmid (3^{\nu-1}m) }(-1)^d\vf\left(\frac{n}{d}\right)\\
&=\frac{1}{3n} \sum_{d=3^\nu q,\, q| m} (-1)^{d}\vf\left(\frac{n}{d}\right)\\
&=\frac{1}{3n} \sum_{ q| m} (-1)^{3^\nu q}\vf\left(\frac{m}{q}\right)=\frac{1}{3n} \sum_{ q| m} (-1)^{ q}\vf\left(\frac{m}{q}\right)\\
&=-\frac{m}{3n}\mathbb{I}_{\ent\setminus2\ent}(m)=-
\frac{1}{3^{1+\nu}}\mathbb{I}_{\ent\setminus2\ent}(m).
\end{align*}
Finally, noting that $n=m\mod 2$, we obtain the following formula for $\eps_n$ which is also valid when $\nu=0$ according to \eqref{eq219}:
\begin{equation}
\eps_n=-\frac{1}{ 3^{1+\nu_3(n)}}\mathbb{I}_{\ent\setminus2\ent}(n).
\end{equation}

Now note that $b_n$ can be written ans follows
\begin{equation}
b_n=\frac{1}{6n}\sum_{d|n}2^d\lambda(n,d)\vf\left(\frac{n}{d}\right)
\end{equation}
with 
\[
\lambda(n,d)=1+3J(n,d)+2K(n,d)
\]
where 
\begin{align*}
J(n,d)&=\begin{cases} 
            1&\textrm{~if $2|n$ and $d|(n/2)$,}\\
            0&\textrm{otherwise}.
       \end{cases}\\
\noalign{\noindent\textrm{and}~}
K(n,d)&=\begin{cases} 
            1&\textrm{~if $3|n$ and $d|(n/3)$,}\\
            0&\textrm{otherwise}.
       \end{cases}
\end{align*}
equivalently
\begin{equation*}
J(n,d)=\mathbb{I}_{2\ent}\left(\frac{n}{d}\right),
\quad\textrm{and}\quad
K(n,d)=\mathbb{I}_{3\ent}\left(\frac{n}{d}\right).
\end{equation*}
Thus
\begin{equation}
\lambda(n,d)=1+3\mathbb{I}_{2\ent}\left(\frac{n}{d}\right)+2\mathbb{I}_{3\ent}\left(\frac{n}{d}\right)
\end{equation}
So, we may write $b_n$ in the following form
\begin{equation}
b_n=\frac{1}{6n}\sum_{d|n}2^d\chi\left(\frac{n}{d}\right)\,\vf\left(\frac{n}{d}\right)
\end{equation}
with $\chi:\ent\to\nat_6$ defined by
\begin{equation}
\chi(k)=(1+3\mathbb{I}_{2\ent}(k)+2\mathbb{I}_{3\ent}(k))=\begin{cases}
1&\textrm{if $\gcd(k,6)=1$,}\\
4&\textrm{if $\gcd(k,6)=2$,}\\
3&\textrm{if $\gcd(k,6)=3$,}\\
6&\textrm{if $\gcd(k,6)=6$.}
\end{cases}
\end{equation}
This can also be written in the form $\chi(k)=(1+\mathbb{I}_{2\ent\setminus3\ent}(k))\gcd(k,6)$,
and the announced expression \eqref{eqexp1} for $K(n)$ is obtained. 
Finally, the  formula $K(n)=\floor{b_n}$, follows from the fact that.
$-\frac{1}{3}\le\eps_n\le 0$.
\end{proof}

We conclude our discussion of the case of necklaces by noting that there are some simple cases where the formula for $K(n)$ is particularly appealing, for example, if $n=p>3$ is  prime, then 
\[
K(p)= \frac{2^p-2}{6p},
\]
and if $\gcd(n,6)=1$, then
\[
K(n)=\floor{\frac{1}{6n}\sum_{d|n}\vf(d) 2^{n/d}}
= \frac{1}{6n}\sum_{d|n}\vf(d)2^{n/d} -\frac{1}{3}.
\] 
\begin{remark}
 If $6$ and $n$ are coprime, then $K(n)$ is related to the number
$N(n,2)$ of $n$-bead necklaces of two colors \eqref{eq10} by the formula \[K(n)=\floor{N(n,2)/6}=(N(n,2)-2)/6.\]
\end{remark}

\begin{remark}
	An equivalent formula for $K(n)$ that does not use the indicator function of the set $2\ent\setminus 3\ent$ is the following
	\[
	K(n)=\floor{\frac{1}{6n}\sum_{d|n}
		\left(1+\frac{4}{3}\cos^2\left(\frac{d\pi}{2}\right)
		\sin^2\left(\frac{d\pi}{3}\right)\right)\gcd(d,6)\,\vf(d) 2^{n/d}}.
	\]
\end{remark}	

Table \ref{table1} lists the first 40 terms of the sequence $(K(n))_{n\ge1}$.

\begin{table}[h!]
\centering
\begin{tabular}{||c|c||c|c||c|c||c|c||}
\hline\rule{0pt}{11pt}
 $n$& $K(n)$& $n$& $K(n)$& $n$& $K(n)$&$n$& $K(n)$\\
\hline\hline
 1 & 0 & 11 & 31 & 21 & 16651 & 31 & 11545611 \\
 2 & 1 & 12 & 64 & 22 & 31838 & 32 & 22371000 \\
 3 & 1 & 13 & 105 & 23 & 60787 & 33 & 43383571 \\
 4 & 2 & 14 & 202 & 24 & 116640 & 34 & 84217616 \\
 5 & 1 & 15 & 367 & 25 & 223697 & 35 & 163617805 \\
 6 & 4 & 16 & 696 & 26 & 430396 & 36 & 318150720 \\
 7 & 3 & 17 & 1285 & 27 & 828525 & 37 & 619094385 \\
 8 & 8 & 18 & 2452 & 28 & 1598228 & 38 & 1205614054 \\
 9 & 11 & 19 & 4599 & 29 & 3085465 & 39 & 2349384031 \\
 10 & 20 & 20 & 8776 & 30 & 5966000 & 40 & 4581315968 \\
 \hline
\end{tabular}\\
$\,$
\caption{List of $K(1),\ldots, K(40)$, which counts colorful necklaces.}
\label{table1}
\end{table}

\section{Counting Colorful Bracelets}\label{sec4}
As we explained before, bracelets are turnover necklaces. It is the action of the group $G'=\mathfrak{S}_3\times \langle r,s\rangle$ on the set
of $n$-periodic colorful sequences $\mathcal{A}_n$ that is considered. 

We are interested in the number of orbits $\mathcal{A}_n/G'$ denoted by $K'(n)$. Again Burnside's Lemma comes to our rescue. We need to determine the numbers $\fn(\sigma,\eps,\ell)$ with $\sigma\in\mathfrak{S}_3$, $\eps\in\{0,1\}$
and $\ell\in\ent/n\ent$, but we have already done this in the case $\eps=0$ in the previous section. 

Further, based Corollary \ref{cor27}, we only need to determine  $\fn(\sigma,1,0)$ and $\fn(\sigma,1,1)$ for $\sigma$ in 
$\{id,\tau_{12},c\}$. This is the object of the next proposition.

\bigskip\goodbreak

\begin{proposition}\label{pr41}$\,$
\begin{enumerate}[label=\textbf{\upshape{(\alph*)}}]
\item 
\begin{enumerate}[label=\textit{\roman*.},leftmargin=0.5cm]
	\item If $n$ is odd then $\fn(id,1,0)=0$, otherwise 
	$\fn(id,1,0)=3\times 2^{n/2}$.
\item $\fn(id,1,1)=0$.
\end{enumerate}
\item 
\begin{enumerate}[label=\textit{\roman*.},leftmargin=0.5cm]
\item $\fn(\tau_{12},1,0) =\alpha_{\floor{(n+1)/2}}/3.$
\item $\fn(\tau_{12},1,1) =\alpha_{\floor{n/2+1}}/3.$
\end{enumerate}
\item $\fn(c,1,0) =\fn(c,1,1)=0.$
\end{enumerate}
\end{proposition}	
\begin{proof}%%%%%%%%%%%%%%
\textbf{(a)} Suppose that $\Fn(id,1,0)\ne\emptyset$ and consider $f\in \Fn(id,1,0)$. Write $n=2m+t$ with $t\in\{0,1\}$. Because $f(k)=f(-k)=f(n-k)$ for every $k$, we conclude by considering $k=m$ that $f(m+t)=f(m)$. But $f(m)\ne f(m+1)$ so we must have $t=0$ and $n=2m$. Now, from the fact that 
$f(2m-k)=f(k)$ for every $k$ we conclude that 
\begin{multline*}
\underbrace{\big(f(0),\ldots,f(m),f(m+1),\ldots,f(2m-1)\big)}_{\textrm{a period of $n=2m$}}=\\
\big(f(0),\ldots,f(m),f(m-1),\ldots,f(1)\big).
\end{multline*}
So, the mapping 
\[
f\mapsto (f(0),f(1),\ldots,f(m))
\]
defines a bijection between $\Fn(id,1,0)$ and the set
\[
\left\{(x_0,\ldots,x_m)\in\nat_3: x_{i+1}\ne x_i,i=0,\ldots,m-1\right\}
\]
Now, $x_0$ may take any one of three possible values and each other $x_i$ has two possible values. So, the cardinality of this set is
$3\times 2^m$. Thus \textbf{(a)} \textit{i.} is proved.

Now suppose that $\Fn(id,1,1)\ne\emptyset$ and consider  $f$ from $ {\Fn(id,1,1)}$. We have $f(-k-1)=f(k)$ for every $k$, in particular, for $k=0$ we get $f(-1)=f(0)$ which is absurd, and \textbf{(a)} \textit{ii.} follows.
%%%%%%%%%%%%%%

\noindent\textbf{(b)} \textit{i.} we  write $\tau$ for $\tau_{12}$. Suppose that $\Fn(\tau,1,0)\ne\emptyset$ and consider $f\in \Fn(\tau,1,0)$. We have
\[\forall\, k\in\ent,\quad f(-k)=\tau(f(k)).	\]
Taking $k=0$ we get $ f(0)=\tau(f(0))$, and this implies that $f(0)=3$.

\begin{itemize}[leftmargin=0.5cm]
	\item If $n=2m$ then $f(m)=f(m-n)=f(-m)=\tau(f(m))$ and consequently $f(m)=3$. The restriction of $f$ to the period $\{-m+1,\ldots,m-1,m\}$ has the form
	\[(\tau(f(m-1)),\ldots,\tau(f(1)),3,f(1),\ldots,f(m-1),3).\]
	So, $f$ is completely determined by the knowledge of $(f(1),\ldots,f(m-1))$ and consequently there is a bijection between
	$\Fix^{(2m)}(\tau,1,0)$ and $\mathcal{A}_m^{3\cdot\cdot1} \cup
	\mathcal{A}_m^{3\cdot\cdot2} $. Thus, by Corollary \ref{cor2}, we have
	\[
	\fn(\tau,1,0)=\frac{\alpha_m}{3}=\frac{1}{3}\alpha_{\floor{(n+1)/2}}.
	\]
	\item If $n=2m+1$, then $f(m+1)=f(m+1-n)=f(-m)=\tau(f(m))$ and consequently $f(m)\ne3$. The restriction of $f$ to the period $\{-m,\ldots,m\}$ takes the form
	\[(\tau(f(m)),\tau(f(m-1)),\ldots,\tau(f(1)),3,f(1),\ldots,f(m))\]
	So, $f$ is completely determined by the knowledge of $(f(1),\ldots,f(m))$ and consequently there is a bijection between
	$\Fix^{(2m+1)}(\tau,1,0)$ and $\mathcal{A}_{m+1}^{3\cdot\cdot1} \cup
	\mathcal{A}_{m+1}^{3\cdot\cdot2} $.  Thus
	\[
	\fn(\tau,1,0)=\frac{\alpha_{m+1}}{3}=\frac{1}{3}\alpha_{\floor{(n+1)/2}}.
	\]
\end{itemize}
%%%%%%%%%%%%%%

\noindent\textbf{(b)} \textit{ii.} Now suppose that $\Fn(\tau,1,1)\ne\emptyset$ and consider $f\in \Fn(\tau,1,1)$. We have
\[\forall\, k\in\ent,\quad f(-k-1)=\tau(f(k))	\]
Taking $k=0$ we get $ f(-1)=\tau(f(0))$, but $f(-1)\ne f(0)$ thus
$f(0)\in\{1,2\}$.
\begin{itemize}[leftmargin=0.5cm]
	\item If $n=2m$, then $f(m-1)=f(m-1-n)=f(-m-1)=\tau(f(m))$  but $f(m-1)\ne f(m)$ thus $f(m-1)\in\{1,2\}$. The restriction of $f$ to the  period $\{-m,\ldots,m-1\}$ takes the form
	\[(\tau(f(m-1)),\ldots,\tau(f(0)),f(0),f(1),\ldots,f(m-1)).\]
	So, $f$ is completely determined by the knowledge of $(f(0),\ldots,f(m-1))$. We can partition the set $\Fix^{(2m)}(\tau,1,1)$
	according to the values taken by  ${(f(0),f(m-1))}$, and we have obvious bijective mappings:
	\begin{align*}
	\Fix^{(2m)}(\tau,1,1)\cap\{f:f(0)=1,f(m-1)=1\}&\to
	\mathcal{A}_{m-1}^{1\cdot\cdot2} \cup
	\mathcal{A}_{m-1}^{1\cdot\cdot3}\\
	\Fix^{(2m)}(\tau,1,1)\cap\{f:f(0)=2,f(m-1)=2\}&\to
	\mathcal{A}_{m-1}^{2\cdot\cdot1} \cup
	\mathcal{A}_{m-1}^{2\cdot\cdot3}\\
	\Fix^{(2m)}(\tau,1,1)\cap\{f:f(0)=1,f(m-1)=2\}&\to
	\mathcal{A}_{m}^{1\cdot\cdot2}\\
	\Fix^{(2m)}(\tau,1,1)\cap\{f:f(0)=2,f(m-1)=1\}&\to
	\mathcal{A}_{m}^{2\cdot\cdot1}
	\end{align*}
	Thus
	\begin{equation*}
	\qquad\fn(\tau,1,1) =\frac{1}{6}(4\alpha_{m-1}+2\alpha_m)
	=\frac{2}{3}\left(2^m-(-1)^m\right)=\frac{1}{3}
	\alpha_{\floor{n/2+1}}.	\end{equation*}
	\item If $n=2m+1$ then $f(m)=f(m-n)=f(-m-1)=\tau(f(m))$ and consequently $f(m)=3$. The restriction of $f$ to the set $\{-m,\ldots,m\}$ takes the form
	\[(\tau(f(m-1)),\ldots,\tau(f(0)),f(0),f(1),\ldots,f(m-1),3).\]
	So, $f$ is completely determined by the knowledge of $(f(0),\ldots,f(m-1))$ and  there is an obvious bijective mapping between
	$\Fix^{(2m+1)}(\tau,1,1)$ and $\mathcal{A}_{m+1}^{1\cdot\cdot3} \cup
	\mathcal{A}_{m+1}^{2\cdot\cdot3} $.  Thus
	\[
	\fn(\tau,1,1)=\frac{\alpha_{m+1}}{3}=\frac{1}{3}
	\alpha_{\floor{n/2+1}}.
	\]
	This concludes the proof of part \textbf{(b)}.
\end{itemize}
%%%%%%%%%%%%%

\noindent\textbf{(c)} First, suppose that $\Fn(c,1,0)\ne\emptyset$ and consider $f\in \Fn(c,1,0)$. We have
$ f(k)=c(f(-k))$ for all $ k\in\ent$.
In particular, $f(0)=c(f(0))$ which is absurd because $c$ has no fixed points.	

Next suppose that $\Fn(c,1,1)\ne\emptyset$ and consider $f\in \Fn(c,1,1)$. We have
$ f(k)=c(f(-k-1))$ for all $ k\in\ent$. 
\begin{itemize}[leftmargin=0.5cm]
	\item If $n=2m+1$ then
	\[f(m)=f(m-n)=f(-m-1)=c^{-1}(f(m)),\] which is absurd because $c$ has no fixed points.
	\item If $n=2m$ then
	\begin{align*}
		f(m)&=f(m-n)=f(-m)=c^{-1}(f(m-1))\\
		&=c^{-1}(f(m-1-n))
		=c^{-1}(f(-m-1))\\
		&=c^{-2}(f(m))
	\end{align*}
	which is also absurd because $c^2$ has no fixed points.	
\end{itemize}			
This achieves the proof of the proposition.	
\end{proof}

Finally we arrive to the main theorem of this section.

\begin{theorem}\label{th316}
	The number of non-equivalent colorful $n$-bead Bracelets with three colors is given by 
		\begin{equation}\label{eq0}
			K'(n)=\frac{K(n)+R(n)}{2}
		\end{equation}
	with
	\begin{equation}\label{eq}
	R(n)=\begin{cases} 2^{n/2-1}&\textrm{~if $n$ is even,}\\
	\frac{1}{3}(2^{(n-1)/2}-(-1)^{(n-1)/2}) &\textrm{~if $n$ is odd.}
	\end{cases}
	\end{equation}
	where $K(n)$ is given by Theorem \ref{th214}.
\end{theorem} 
\begin{proof}
We only need to put things together. We know that
\[
K'(n)=\frac{1}{12n}\left(\sum_{(\sigma,j)\in \mathfrak{S}_3\times\ent/n\ent}\fn(\sigma,0,j)
+\sum_{(\sigma,j)\in \mathfrak{S}_3\times\ent/n\ent}\fn(\sigma,1,j)
\right).
\]
Thus $K'(n)=(K(n)+R(n))/2$ with 
\begin{align*}
R(n)&=\frac{1}{6n}\sum_{(\sigma,j)\in \mathfrak{S}_3\times\ent/n\ent}\fn(\sigma,1,j)\\
&=\frac{1}{6n}\left(\sum_{\genfrac{}{}{0pt}{2}{\sigma\in \mathfrak{S}_3}{
	0\le 2j\le n-1}}\fn(\sigma,1,{2j})+
\sum_{\genfrac{}{}{0pt}{2}{\sigma\in \mathfrak{S}_3}{
	0\le 2j+1\le n-1}}\fn(\sigma,1,2j+1)
\right)\\
&=\frac{1}{6n}\left(\sum_{\genfrac{}{}{0pt}{2}{\sigma\in \mathfrak{S}_3}{
		0\le 2j\le n-1}}\fn(\sigma,1,0)+
\sum_{\genfrac{}{}{0pt}{2}{\sigma\in \mathfrak{S}_3}{
		0\le 2j+1\le n-1}}\fn(\sigma,1,1)
\right)\\
&=\frac{1}{6n}\left(\floor{\frac{n+1}{2}}
\sum_{\sigma\in \mathfrak{S}_3}\fn(\sigma,1,0)+
\floor{\frac{n}{2}}
\sum_{\sigma\in \mathfrak{S}_3}\fn(\sigma,1,1)
\right)
\end{align*}
where we used Corollary \ref{cor27}. Now using Proposition
 \ref{pr41} we get 
\begin{align}
\sum_{\sigma\in \mathfrak{S}_3}\fn(\sigma,1,0)&=
\fn(id,1,0)+3\fn(\tau_{12},1,0)\notag\\
&=\begin{cases}
3\times 2^m +\alpha_{m} &\textrm{if $n=2m$,}\\
\alpha_{m+1}	        &\textrm{if $n=2m+1$,}\\
\end{cases} 
\end{align}
and
\begin{align}
	\sum_{\sigma\in \mathfrak{S}_3}\fn(\sigma,1,1)&=
	3\fn(\tau_{1,2},1,1)\notag\\
	&=\begin{cases}
		\alpha_{m+1} &\textrm{if $n=2m$,}\\
		\alpha_{m+1} &\textrm{if $n=2m+1$.}\\
	\end{cases} 
\end{align}
Replacing in the expression of $R(n)$ we obtain
\[R(n)=\begin{cases}
 2^{m-1} &\textrm{if $n=2m$,}\\
\alpha_{m+1}/6 &\textrm{if $n=2m+1$.}\\
\end{cases}
\]
and the announced result follows.
\end{proof}	

Table \ref{table2} lists the first 40 terms of the sequence $(K'(n))_n$.

\begin{table}[h!]
	\centering
	\begin{tabular}{||c|c||c|c||c|c||c|c||}
		\hline\rule{0pt}{11pt}
		$n$& $K'(n)$& $n$& $K'(n)$& $n$& $K'(n)$&$n$& $K'(n)$\\
		\hline\hline
	1 & 0 & 11 & 21 & 21 & 8496 & 31 & 5778267 \\
	2 & 1 & 12 & 48 & 22 & 16431 & 32 & 11201884 \\
	3 & 1 & 13 & 63 & 23 & 30735 & 33 & 21702708 \\
	4 & 2 & 14 & 133 & 24 & 59344 & 34 & 42141576 \\
	5 & 1 & 15 & 205 & 25 & 112531 & 35 & 81830748 \\
	6 & 4 & 16 & 412 & 26 & 217246 & 36 & 159140896 \\
	7 & 3 & 17 & 685 & 27 & 415628 & 37 & 309590883 \\
	8 & 8 & 18 & 1354 & 28 & 803210 & 38 & 602938099 \\
	9 & 8 & 19 & 2385 & 29 & 1545463 & 39 & 1174779397 \\
	10 & 18 & 20 & 4644 & 30 & 2991192 & 40 & 2290920128\\
		\hline
	\end{tabular}\\
$\,$
	\caption{List of $K'(1),\ldots, K'(40)$, which counts colorful bracelets.}
	\label{table2}
\end{table}

\begin{remark} Although $K(n) \le K'(n)$ for all $n\ge1,$ a surprising fact about $(K(n))_{n\ge1}$ and $(K'(n))_{n\ge1}$ is that they coincide for the first $8$ values!.
\end{remark}	

\begin{remark}
The equality $2K'(n)=K(n)+R(n)$ and the easy-to-prove fact that
$R(n)=n\mod 2$ for $n\ge3$, allow us to find the parity pattern of the $K(n)$'s. The fact that $K(n)=n\mod 2$ for $n\ge3$ seems difficult to prove directly.
\end{remark}

%%%%%%%%%%%%%%%%%%%%%%%%%%%%%%%%%%%%%%%%
%%%%%%%%%%%%%%%%%%%%%%%%%%%%%%%%%%%%%%%%
\section{Related Combinatorial Sequences}\label{sec6}
%%%%%%%%%%%%%%%%%%%%%%%%%%%%%%%%%%%%%%%%

Colorful necklaces or bracelets with $n$ beads and two colors are easy to determine. There are none when $n$ is odd and just one equivalence class when $n$ is even. Thus the sequences
$(K^*(n))_{n\ge1}$ and $(K^{\prime*}(n))_{n\ge1}$
defined by
\begin{equation}
 K^*(n)=K(n)-\frac{1+(-1)^n}{2},\quad\textrm{and}\quad
 K^{\prime*}(n)=K'(n)-\frac{1+(-1)^n}{2}
\end{equation}
represent  the number of non-equivalent colorful necklaces in $n$ beads with exactly $3$ colors and  the number of non-equivalent colorful bracelets in $n$ beads with exactly $3$ colors, respectively. Both sequences $(K^*(n))_{n\ge1}$, and $(K^{\prime*}(n))_{n\ge1}$ are currently not recognized by the OEIS.

Further, if we are interested in periodic colorful sequences of \textit{exact} period $n$ in at most $3$ colors then the number $\widetilde{K}(n)$ of non-equivalent such sequences assuming that reversing is not allowed is given by
\begin{equation*}
	\widetilde{K}(n)=\sum_{d|n}\mu\left(\frac{n}{d}\right) K(d)
\end{equation*}	
where $\mu$ is the well known Moebius function. Indeed, this follows from the classical result
\cite[Theorem 1.5]{rose}, because clearly $K(n)=\sum_{d|n}	\widetilde{K}(d)$.

\noindent OEIS recognizes  $(\widetilde{K}(n))_{n\ge1}$ as the ``Number of  Zn\,S  polytypes that repeat after $n$ layers''
\href{https://oeis.org/A011957}{A011957}. 

Similarly, if we are interested in periodic colorful sequences of \textit{exact} period $n$ in at most $3$ colors then the number $\widetilde{K}'(n)$ of non-equivalent such sequences assuming that reversing is allowed is given by
\begin{equation*}
	\widetilde{K}'(n)=\sum_{d|n}\mu\left(\frac{n}{d}\right) K'(d).
\end{equation*}	
OEIS recognizes  $(\widetilde{K}'(n))_{n\ge1}$ as the ``Number of Barlow packings that repeat after exactly $n$ layers''
\href{https://oeis.org/A011768}{A011768}. 

%%%%%%%%%%%%%%%%%%%%%%%%%%%%%%%%%%%%%%%%%%%%%%%%%%%%

%%%%%%%%%%%%%%%%%%%%%%%%%%%%%%%%%%%%%%%%
%%%%%%%%%%%%%%%%%%%%%%%%%%%%%%%%%%%%%%%%
\section{Future Research}\label{secFR}
%%%%%%%%%%%%%%%%%%%%%%%%%%%%%%%%%%%%%%%%

This paper has counted non-equivalent colorful necklaces and non-equivalent colorful bracelets in $n$ beads with $3$ colors.
An open problem is to count  non-equivalent colorful necklaces and colorful bracelets in $n$ beads with $c\ge4$ colors.

\end{document}